\documentclass{article}

\usepackage[utf8]{inputenc}
\usepackage[T1]{fontenc}
\usepackage[english]{babel}

\usepackage{graphicx}
\graphicspath{ {./images/} }
\usepackage{amsmath}
\usepackage{amsfonts}
\usepackage{amsthm}
\usepackage{amssymb}
\usepackage{mathtools}
\usepackage{xcolor}
\usepackage{xfrac} 
\usepackage{commath} 
\usepackage{tikz}
\usepackage{appendix} 
\usepackage{diagbox} 
\usepackage{subcaption} 
\usepackage{float}
\usepackage{lipsum}

\makeatletter
\g@addto@macro\@floatboxreset\centering
\makeatother

\newtheorem{remark}{Remark}
\newtheorem{theorem}{Theorem}
\newtheorem{lemma}{Lemma}

\usepackage[%
backend=biber,
sortlocale=auto,
hyperref,
style=numeric,
maxbibnames=99
]%
{biblatex}

\newcommand{\I}{\mathcal{I}}
\newcommand{\E}{E}

\DeclareMathOperator{\cl}{cl} 	
\DeclareMathOperator{\interior}{int} 
\renewcommand{\div}[1]{\nabla_{#1} \cdot} 
\newcommand{\dd}{{\rm d}}

\usepackage[mathscr]{euscript}
\DeclareSymbolFont{rsfs}{U}{rsfs}{m}{n}
\DeclareSymbolFontAlphabet{\mathscrsfs}{rsfs}

\newcommand{\R}{\mathbb{R}}								
\newcommand{\N}{\mathbb{N}}								
\let\div\relax 
\DeclareMathOperator{\div}{div}							
\newcommand{\scalar}[2]{( #1, #2 )} 			
\newcommand{\paren}[1]{\left( #1 \right)} 						
\newcommand{\Tau}{\mathcal{T}}

\newcommand{\eqrel}[2]{\stackrel{\scriptscriptstyle{\eqref{#1}}}{#2}}

\title{Well-Posedness and Finite Element Approximation of Mixed Dimensional Partial Differential Equations}
\author{F. Hellman\footnotemark[1], A. Målqvist\footnotemark[1], M. Mosquera\footnote {Department of Mathematical Sciences, Chalmers University of Technology and University of Gothenburg, 412 96 Göteborg, Sweden}}
\date{}

\begin{document}

\maketitle

\begin{abstract}
We consider a mixed dimensional elliptic partial differential equation posed in a bulk domain with a large number of embedded interfaces. In particular, we study well-posedness of the problem and regularity of the solution. We also propose a fitted finite element approximation and prove an a priori error bound.  For the solution of the arising linear system we propose and analyze an iterative method based on subspace decomposition. Finally, we present numerical experiments and achieve rapid convergence using the proposed preconditioner, confirming our theoretical findings.
\end{abstract}

\section{Introduction}
Numerical simulation of diffusion processes in heterogeneous materials is computationally challenging since the data variation needs to be resolved. Thin structures like cracks, fractures, and reinforcements are particularly difficult to handle. It is often advantageous to instead model thin structures as lower dimensional interfaces, embedded in a bulk domain. Mathematically this results in a mixed dimensional partial differential equation (PDE) where the solution has bulk and interface components that are coupled weakly. The aim of the paper is to study well-posedness and regularity of mixed dimensional PDEs, construct and analyze a finite element method for solving the problem, and to develop and implement a preconditioner for efficient numerical solution of the resulting linear system.

Numerical solution to PDEs posed on surfaces is a well established field. An early contribution to finite element approximation of the Laplace--Beltrami equation on surfaces is due to Dziuk in \cite{Dz88}, where a finite element approximation is constructed on a polyhedral approximation of the surface. Other approaches include trace based methods as in \cite{BuHaLa19, BCHLM15, OlReGr09,}, where functions in the surface are represented by traces of functions in the ambient space of higher dimension. The review articles \cite{BoDeoNo20, DzEl13} survey several methods for solving PDEs on surfaces. Finite element methods for coupled bulk and interface problems are also well studied, see e.g.~\cite{ElRa12,MaJaRo05}.

Flow in porous media is probably the most prominent application for mixed dimensional PDEs, see \cite{AGPR19,BNY18,PFS17,FK18,MaJaRo05} and references therein. These methods are fitted, meaning that the full geometry is resolved by the mesh. The bulk is three dimensional and the interfaces are two dimensional surfaces representing fractures. There is also related work treating three dimensional bulk domains with one dimensional embedded structures. Such models include blood vessels embedded in tissue, see \cite{FKOWW22}.  

In this paper we consider an elliptic mixed dimensional model problem with a large number of interfaces. The interface model is inspired by the works \cite{BNY18, NoBoFuKe19} where a general framework for multidimensional representations of fractured domains is developed, and the work \cite{MaJaRo05} where Robin-type couplings between the bulk and interface are studied. In the bulk and interfaces, we consider Poisson's equation in $d$ and $d-1$ dimensions, respectively. In both cases, we use its primal form. 
We show that the model is well-posed and pay particular attention to problems with a large number of embedded interfaces and how that affects the coercivity bound of the variational formulation of the problem.
We further formulate a fitted finite element method and prove a priori error bounds. Additionally, we propose a domain decomposition method based on subspace correction allowing for efficient numerical solution of problems with complex coupled interface structures. In this part we take inspiration from recently developed subspace decomposition algorithms for spatial network models, see \cite{GoHeMa22, KY16}. We formulate the Schur complement on the union of interfaces and introduce coarse subspaces, using coarse finite element spaces in the bulk interpolated onto the interfaces. Finally, we present numerical examples in two spatial dimensions together with regularity analysis.


Outline: In Section~\ref{sec:problem} we present the problem formulation and in Section~\ref{sec:fem} we formulate a fitted finite element method and derive an a priori error bound. Section \ref{sec:iterative} is devoted to an iterative method for efficient solution of the resulting linear system. Finally, in Section \ref{sec:num_ex} we present numerical results.

\section{Problem formulation}
\label{sec:problem}
We consider Poisson's equation in the bulk as well as in the
interfaces and couple the solutions in bulk and interfaces by Robin
conditions. We are interested in the case when there is a large number
of interfaces in the domain. Next
follows a description of the geometrical setting to accomodate for
this partitioning. Then, we present the model problem on weak form and
establish its well-posedness.

\subsection{Mixed dimensional geometry}
We consider an open and connected domain $\Omega$ of $\R^d$ ($d=2$ or
$3$) that is partitioned into $d+1$ (not necessarily connected)
subdomains of different dimensionalities:
\begin{equation}
  \Omega = \Omega^0 \cup \cdots \cup \Omega^d
\end{equation}
where each $\Omega^c$ ($0 \le c \le d$ is the codimension)
in turn is composed into a finite number of disjoint and connected
\emph{subdomain segments} defined by a set $L_c$,
\begin{equation}
  \Omega^c = \bigcup_{\ell \in L_c} \Omega^c_\ell,
\end{equation}
each of which is a planar hypersurface of codimension $c$. 

In the following, we put some requirements on this decomposition of
$\Omega$. We define topological subspaces $X^c$ of $\R^d$ for
subdomains in codimension $c$ and topological properties such as
dense, closed, open, and operators such as closure ($\cl$), boundary
($\partial$) and interior ($\interior$) for subdomains of codimension
$c$ should be interpreted in the subspace topology of $X^c$. For a set $\omega$, 
the notation $\overline \omega$ denotes the closure of $\omega$ in
$\R^d$.

For codimension $c = 0$, we let $X^0 = \Omega$, while for $c > 0$, we
define the subspaces $X^c = X^{c-1} \setminus \Omega^{c-1}$.
We assume that for all $\ell \in L_c$, the subdomain segment
$\Omega^c_\ell$ is open and that their union $\Omega^c$ is
dense. 
We add the requirement that either
$\Omega^{c+1}_j \subseteq \partial \Omega^{c}_i$ or
$\Omega^{c+1}_j \cap \Omega^{c}_i = \emptyset$.  This makes it possible
to define the adjacency relation ${\E}_{c}$ between subdomain
segments $i \in L_{c}$ and $j \in L_{c+1}$  by
\begin{equation}
  \begin{aligned}
    (i, j) \in \E_{c}
    \qquad &\text{if } \Omega^{c+1}_{j} \subseteq \partial \Omega^{c}_{i}.
  \end{aligned}
\end{equation}
This requirement admits a dense partitioning of a subdomain segment
by subdomain segments of one codimension higher. Thus subdomain segment
boundary integrals can be expressed as sums of integrals over other
subdomain segments.

We focus entirely on the subdomains of codimension 0, 1 and 2. To
simplify notation, we introduce $I = L_0$, $J = L_1$, and $K =
L_2$. To reduce the complexity of the model, we have chosen to work in
a simplified setting: We assume that all subdomains segments
$\Omega^c_\ell$ are planar and polyhedrons, polygonals, lines or
points, whichever applies in their respective dimension, and that they
have Lipschitz boundary.  Lipschitz boundary rules out slit domains
and means that, for example, an interface cannot end inside a bulk
without connecting to another interface.
These assumptions can be
relaxed at the cost of a more detailed treatment of the interfaces as
is done in \cite{BNY18, NoBoFuKe19}. To handle curved interfaces, it
is possible to consider the planar setting here as an approximation of
the curved case. For sufficiently smooth curved interfaces, lift
operators on both bulk and interfaces exist that enable an estimate of
the error due to such a variational crime, see \cite{ElRa12}.

In Figure \ref{fig:domain_rules}, we see an example of a domain fulfilling our requirements.

\begin{figure}[h!]
	\includegraphics[width=.7\textwidth]{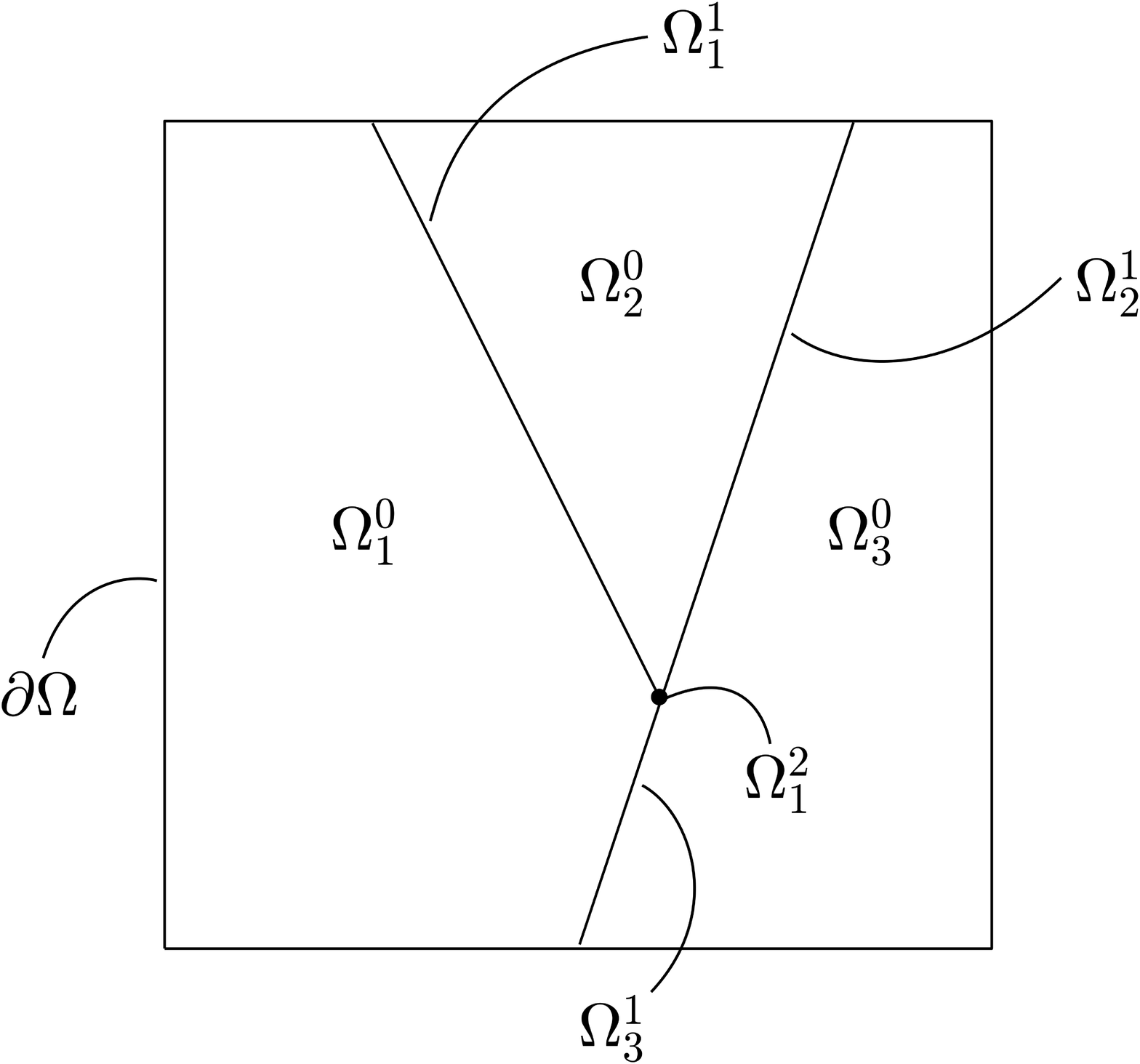}
	\caption{An example of a domain $\Omega$ fulfilling our requirements.}\label{fig:domain_rules}
\end{figure}

\subsection{Model problem}
\label{sec:weak_formulation}

Denote the $L^2$-inner product and norm
$(u, v)_\omega = \int_\omega uv\,\dd x$,
$\|v\|_{L^2(\omega)} = (u, v)_\omega^{1/2}$, and let $L^2(\omega)$ and
$H^1(\omega)$ be the usual Sobolev spaces of $L^2$-integrable
functions and weak derivatives. If $\omega$ is a subdomain segment of
codimension greater than 0, then we interpret $H^1(\omega)$ as $H^1(\widehat \omega)$ after a
rigid coordinate transformation of $\omega$ to a subset
$\widehat \omega \subset \R^{d-c}$ where the gradient operator,
Poincaré inequalites, Green's theorem, and trace theorems in this
lower dimensional space $\R^{d-c}$ can be applied. For clarity, the
symbol $\nabla_\tau$ is used to denote the gradient operator on
subdomain segments of codimension 1.

For functions in the bulk of the domain, we define
$$V^0 = H^1(\Omega^0) = \prod_{i\in I} H^1(\Omega^0_i).$$ 
Occasionally, we refer to components $v^0_i$ of $v^0 \in V^0$ defined as
the restriction of $v^0$ to $\Omega^0_i$. Note that the domain
$\Omega^0$ consists of the disconnected union of $\Omega^0_i$ for
$i \in I$ and no continuity between the subdomain segments is
assumed. Further,
$V^0_0 = \{ v \in V^0 \,:\, v|_{\partial \Omega} = 0 \}$ is the space
of functions satisfying homogeneous Dirichlet boundary
conditions.

For the functions in codimension 1, we define $V^1$ as follows. The
functions in $V^1$ are strongly enforced to be continuous over common
domains of codimension 2. Let
$V^1_{\text b} = H^1(\Omega^1) = \prod_{j\in J} H^1(\Omega^1_j)$ be a
``broken'' space and $v^1 = (v^1_j)_{j\in J} \in V^1_{\text b}$ be any
function with each $v^1_j$ being the restriction of $v^1$ to
$\Omega^1_j$. Then we define the constrained space
\begin{equation}
  V^1 = \{ v^1 \in V^1_{\text b} \,:\, v^1_j|_{\Omega^2_k} - v^1_{j'}|_{\Omega^2_k} = 0 \text{ for all pairs } (j, k), (j', k) \in {\E_1}\}.
\end{equation}
We note that $V^1$ is a Hilbert space, since it
is the kernel of the bounded and linear trace operator on the Hilbert
space $V^1_{\text b}$. Finally, we define
$V^1_0 = \{ v \in V^1 \,:\, v|_{\partial \Omega} = 0 \}$, and $V = V^0 \times V^1$ and $V_0 = V^0_0 \times V^1_0$ with norm $$\|v\|_V = \left(\|v^0\|_{H^1(\Omega^0)}^2 + \|v^1\|_{H^1(\Omega^1)}^2\right)^{1/2}.$$

In the bulk domains $i \in I$, let $A_i \in L^{\infty}(\Omega^0_i)$
with bounds
$0 < \underline{\alpha} \le A_i \le \overline{\alpha} < \infty$.
In the interfaces $j \in J$, let
$A_j \in L^{\infty}(\Omega^1_j)$ with bounds
$\underline{\alpha} \le A_j \le \overline{\alpha}$, and
$B_j \in L^{\infty}(\Omega^1_j)$ with bounds
$0 < \underline{\beta} \le B_j \le \overline{\beta} < \infty$.
With the definition of the symmetric
bilinear form $a$
\begin{equation*}
  \begin{aligned}
    a(v, w) =& \sum_{i \in I} \scalar{A_i \nabla v^0_i}{\nabla w^0_i}_{\Omega^0_i} + \sum_{j \in J} \scalar{A_j \nabla_\tau v^1_j}{\nabla_\tau w^1_j}_{\Omega^1_j} \\
    &+ \sum_{(i, j) \in \E_0} \scalar{B_j (v^0_i - v^1_j)}{w^0_i - w^1_j}_{\Omega^1_j}, \\
  \end{aligned}
\end{equation*}
and a right hand side functional $F \in V_0^*$ (in the dual space of
$V_0$), the model problem can be stated as to find $u \in V_0$ such
that for all $v \in V_0$,
\begin{equation}
  \label{eq:weak_problem}
  \begin{aligned}
    a(u,v) = F(v).
  \end{aligned}
\end{equation}
The bulk and the interfaces are coupled by a Robin-type boundary
condition, continuity of the solution between the interface segments
is enforced, and a Kirchoff's law type equation for the flux between
interface segments applies. Non-homogeneous Dirichlet boundary
conditions defined on $\partial \Omega$ can be handled by picking any
function $g \in V$ satisfying the boundary conditions, setting
$F(v) = a(-g, v)$, solving \eqref{eq:weak_problem}, and obtaining the
solution as $u + g$.

\begin{remark}[Bulk-interface coupling]
  For the bulk-interface coupling, we use boundary conditions III as
  presented in \cite{MaJaRo05} in a porous media flow setting. The
  interface parameters $A_j$ and $B_j$ translate to tangential and
  normal interface permeabilities $K^1_\tau$ and $K^1_n$, and thickness
  $t$ of the interface as follows: $A_j = K^1_\tau t$ and
  $B_j = K^1_n/t$.
\end{remark}

If the solution is sufficiently smooth, and if
$f_i \in L^2(\Omega^0_i)$, $f_j \in L^2(\Omega^1_j)$, and
\begin{equation*}
  F(w) = \sum_{i \in I} \scalar{f_i}{w^0_i}_{\Omega^0_i} + \sum_{j \in J} \scalar{f_j}{w^1_j}_{\Omega^1_j},
\end{equation*}
then the weak formulation in \eqref{eq:weak_problem} corresponds to the
strong problem of finding $u \in V_0$ such that
\begin{equation*}
  \begin{aligned}
    - \div A_i \nabla u^0_i &= f_i && \text{ in } \Omega^0_i, \\
    - \div_{\tau} A_j \nabla_{\tau} u^1_j - \sum_{i \,:\, (i,j) \in E_0} B_j(u^0_i - u^1_j) &= f_j&& \text{ in } \Omega^1_j, \\
  \end{aligned}
\end{equation*}
with boundary conditions
\begin{equation*}
  \begin{aligned}
    n_{\Omega^0_i} \cdot A_i \nabla u^0_i + B_j u^0_i & = B_j u^1_{j} && \text{ on } \partial \Omega^0_i \cap \Omega^1_j, \\
    u^1_j & = u^1_{j'} && \text{ on } \partial \Omega^1_j \cap \partial \Omega^1_{j'},\\
    \sum_{\{j\,:\,(j, k) \in \E_1\}} n_{\Omega^1_{j}} \cdot A_{j} \nabla_{\tau} u^1_{j} & = 0 && \text{ on } \partial \Omega^1_{j} \cap \Omega^2_k.
 \end{aligned}
\end{equation*}
Here $n_{\omega}$ is the outward normal defined on the boundary of the subdomains $\omega$.

\subsection{Well-posedness}
\label{sec:well_posedness}

To prove existence and uniqueness of solution for the problem
\eqref{eq:weak_problem} we prove that $a$ is coercive and bounded
on $V_0$ and apply the Lax--Milgram theorem. Coercivity cannot be shown
directly by a Poincaré or Friedrichs inequality since the subdomain
segments are disconnected and not all of them necessarily intersect
the domain boundary. We establish coercivity by an iterative
procedure, starting to bound norms over subdomain segments far from
the domain boundary in norms over subdomain segments and proceed until
we reach subdomain segments for which a Friedrichs inequality can be
used.

\begin{lemma}
  \label{lem:coercivity}
  The bilinear form $a$ is coercive on $V_0$, i.e.\ there is a $C > 0$ such that for all $v \in V_0$,
  \begin{equation*}
    a(v,v) \ge C (\underline \alpha^{-1} + \underline \beta^{-1})^{-1} \|v\|_V^2,
  \end{equation*}
  and $C$ depends only on the geometry of the subdomain segments $\{\Omega^c_\ell\}_{c, \ell}$.
\end{lemma}
\begin{proof}
In the following, $C$ denotes a constant that is independent of
the arbitrary function $v = (v^0, v^1) \in V_0 = V^0_0 \times V^1_0$. The value of the constant is
generally not tracked between inequalities.
Let $i_0 \in I$ be a subdomain segment that intersects the domain
boundary so that there is a Friedrichs inequality
\begin{equation}
  \|v^0_{i_0}\|_{L^2(\Omega^0_{i_0})} \le C |v^0_{i_0}|_{H^1(\Omega^0_{i_0})}.
\end{equation}

We note that $G = (I, J, \E_0)$ defines a bipartite undirected graph,
where the subdomain segments of codimensionality 0 and 1 ($I$ and $J$)
constitute the vertices and the edges $\E_0$ are the adjacency
relations between subdomain segments. It is bipartite because the edges
connect vertices in $I$ only with vertices in $J$ and vice versa.
Since $\Omega$ is connected,
$i_0$ is reachable from all vertices in $G$. Then it is possible to
define a walk through $G$ beginning at an arbitrary $i_N \in I$ and
ending at $i_0$. We express the walk as the sequence of vertices it
visits by $(i_N, j_{N}, i_{N-1}, j_{N-1},\ldots, j_1, i_0)$ and
require that all vertices in $G$ are visited at least once by the
walk. We allow the same vertex to be visited multiple times to
guarantee the existence of such a sequence.

We make use of the following
inequality, proven in e.g.\ \cite[eq.\ (5.3.3)]{BrennerScott}, valid
for any $v^0_i \in H^1(\Omega^0_i)$ and $j$ such that $(i, j) \in E_0$,
\begin{equation}
  \label{eq:brenner_scott}
  \begin{aligned}
    \| v^0_i \|_{L^2(\Omega^0_i)} \leq C \left( \left|\int_{\Omega^1_j} v^0_i\, \dd s\right| + | v^0_i |_{H^1(\Omega^0_i)}\right).
  \end{aligned}
\end{equation}
Further, we note that by Hölder's inequality, we have 
\begin{equation}
\left|\int_{\Omega^1_j} v^0_i\, \dd s\right| \le |\Omega^1_j|^{1/2} \|v^0_i\|_{L^2(\Omega^1_j)},
\end{equation}
which together with \eqref{eq:brenner_scott} gives
\begin{equation}
  \label{eq:inverse_trace}
  \begin{aligned}
    \| v^0_i \|_{L^2(\Omega^0_i)} 
    \leq C \left( \| v^0_i \|_{L^2(\Omega^1_j)} + | v^0_i |_{H^1(\Omega^0_i)}\right).
  \end{aligned}
\end{equation}
The following trace inequality will be used, valid for $v^0_i \in H^1(\Omega^0_i)$ for which $(i, j) \in E_0$,
\begin{equation}
  \label{eq:trace}
  \begin{aligned}
    \| v^0_i \|_{L^2(\Omega^1_j)} \le C \left(\| v^0_i \|_{L^2(\Omega^0_i)} + | v^0_i |_{H^1(\Omega^0_i)}\right).
  \end{aligned}
\end{equation}

For $i \in I$, we have
\begin{equation}
  \label{eq:zero}
  \begin{aligned}
    \| v^0_i \|^2_{H^1(\Omega^0_i)} &= | v^0_i |^2_{H^1(\Omega^0_i)} + \| v^0_i \|^2_{L^2(\Omega^0_i)} \\
    & \eqrel{eq:inverse_trace}{\le} C\left(| v^0_i |^2_{H^1(\Omega^0_i)} + \| v^0_i \|^2_{L^2(\Omega^1_j)}\right) \\
    & \le C\left(| v^0_i |^2_{H^1(\Omega^0_i)} + \| v^0_i - v^1_j \|^2_{L^2(\Omega^1_j)} + \| v^1_j \|^2_{L^2(\Omega^1_j)}\right).
  \end{aligned}
\end{equation}

We also have,
\begin{equation}
  \label{eq:one}
  \begin{aligned}
    \| v^1_j \|^2_{L^2(\Omega^1_j)} & \le 2\| v^1_j - v^0_i \|^2_{L^2(\Omega^1_j)} + 2\| v^0_i \|^2_{L^2(\Omega^1_j)} \\
    & \eqrel{eq:trace}{\le} 2\| v^1_j - v^0_i \|^2_{L^2(\Omega^1_j)} + C\| v^0_i \|^2_{H^1(\Omega^0_i)}.
  \end{aligned}
\end{equation}
Since the sequence $(i_N, j_N, \ldots, i_0)$ contains all elements in
both $I$ and $J$, we can rewrite the $V$-norm of $v$ as follows,
\begin{equation}
  \begin{aligned}
    \|v\|_V^2 & = \sum_{i \in I} \|v^0_i\|_{H^1(\Omega^0_i)}^2 + \sum_{j \in J} \|v^1_j\|_{H^1(\Omega^1_j)}^2  \\
    & \le  \|v^0_{i_N}\|_{H^1(\Omega^0_{i_N})}^2 + \|v^1_{j_N}\|_{H^1(\Omega^1_{j_N})}^2 + \|v^0_{i_{N-1}}\|_{H^1(\Omega^0_{i_{N-1}})}^2 + \cdots + \|v^0_{i_{0}}\|_{H^1(\Omega^0_{i_{0}})}^2\\
    & \eqrel{eq:zero}{\le} C \left( |v^0_{i_N}|_{H^1(\Omega^0_{i_N})}^2 + \|v^0_{i_N} - v^1_{j_N}\|_{L^2(\Omega^1_{j_N})}^2\right) + (C + 1)\|v^1_{j_N}\|_{H^1(\Omega^1_{j_N})}^2 + {} \\
    & \phantom{{}\le{}} \|v^0_{i_{N-1}}\|_{H^1(\Omega^0_{i_{N-1}})}^2 + \cdots + \|v^0_{i_{0}}\|_{H^1(\Omega^0_{i_{0}})}^ 2\\
    & \eqrel{eq:one}{\le} C \left( |v^0_{i_N}|_{H^1(\Omega^0_{i_N})}^2 + \|v^0_{i_N} - v^1_{j_N}\|_{L^2(\Omega^1_{j_N})}^2\right) + (C + 1)|v^1_{j_N}|_{H^1(\Omega^1_{j_{N}})}^2 + {}\\
    & \phantom{{}\le{}} 2(C+1)\|v^1_{j_N} - v^0_{i_{N-1}}\|_{L^2(\Omega^1_{j_N})}^2 + (C(C + 1) + 1)\|v^0_{i_{N-1}}\|_{H^1(\Omega^0_{i_{N-1}})}^2 + {} \\
    & \phantom{{}\le{}} \|v^1_{j_{N-1}}\|_{H^1(\Omega^1_{j_{N-1}})}^2 + \cdots + \|v^0_{i_{0}}\|_{H^1(\Omega^0_{i_{0}})}^2.\\
  \end{aligned}
\end{equation}
This procedure is then iterated, using \eqref{eq:zero} and
\eqref{eq:one}. For the last term
$\|v^0_{i_{0}}\|_{H^1(\Omega^0_{i_{0}})}^2$, we use Friedrichs
inequality to bound it by $C |v^0_{i_{0}}|_{H^1(\Omega^0_{i_0})}$.
The walk visits the vertices and edges in the graph at most $N$ times,
and since we have iterated the steps $N$ times we get
\begin{equation}
  \begin{aligned}
    \label{eq:after_walk}
    \|v\|_V^2 & \le NC^{N}\left(\sum_{i \in I} |v^0_i|_{H^1(\Omega^0_i)}^2 + \sum_{(i, j) \in E_0} \|v^0_i - v^1_j\|_{L^2(\Omega^1_j)}^2 + \sum_{j \in J} |v^1_j|_{H^1(\Omega^1_j)}^2\right),
  \end{aligned}
\end{equation}
where we have added the terms for the edges from $E_0$ that were not
part of the walk. The constant $N$ is the length of the walk and
depends on how the subdomain segments are connected with each other
and is thus also a constant depending on the geometry of the
problem. In the following, we do not track $N$ and let $C$ absorb
it.

Finally, to obtain the coercivity bound, we use that $\underline \alpha \le A_i$, $\underline \alpha \leq A_j$ and $\underline \beta \le B_j$ for all
$i \in I$ and $j \in J$, and obtain from \eqref{eq:after_walk} and the definition of $a$ that
\begin{equation*}
  \begin{aligned}
    \|v\|_V^2 & \le C (\underline \alpha^{-1} + \underline \beta^{-1})a(v, v).
  \end{aligned}
\end{equation*}
\end{proof}

\begin{lemma}
  \label{lem:boundedness}
  The bilinear form $a$ is bounded on $V$, i.e.\ there exists a $C < \infty$ such that for all $v \in V$,
  \begin{equation*}
    a(v, w) \le C (\overline{\alpha} + \overline{\beta}) \| v \|_V \, \| w\|_V,
  \end{equation*}
  and $C$ depends only on the geometry of the subdomain segments $\{\Omega^c_\ell\}_{c, \ell}$.
\end{lemma}
\begin{proof}
Studying the first two sums of $a$, we note that
\begin{equation}
  \label{eq:bdd_a}
  \begin{aligned}
    &\sum_{i \in I} \scalar{A_i \nabla v^0_i}{\nabla w^0_i}_{\Omega^0_i} + \sum_{j \in J} \scalar{A_j \nabla_\tau v^1_j}{\nabla_\tau w^1_j}_{\Omega^1_j}  \\
    &\qquad{} \leq \overline{\alpha} \paren{ \| v^0 \|_{H^1(\Omega^0)} \| w^0 \|_{H^1(\Omega^0)} +
      \| v^1\|_{H^1(\Omega^1)} \| w^1 \|_{H^1(\Omega^1)}} \\
    &\qquad{} \leq \overline{\alpha} \paren{ \| v^0 \|_{H^1(\Omega^0)} + \| v^1 \|_{H^1(\Omega^1)}}
      \paren{\| w^0 \|_{H^1(\Omega^0)} + \| w^1 \|_{H^1(\Omega^1)}} \\
    &\qquad{} \leq 2 \overline{\alpha} \, \| v \|_V \, \| w\|_V.
  \end{aligned}
\end{equation}
The trace inequality from equation \eqref{eq:trace} gives us that for $(i, j) \in \E_0$,
\begin{equation*}
  \begin{aligned}
    \|v^0_i - v^1_j\|_{L^2(\Omega^1_j)} & \le \|v^0_i\|_{L^2(\Omega^1_j)} + \|v^1_j\|_{L^2(\Omega^1_j)} \\
     & \le \|v^1_j\|_{L^2(\Omega^1_j)} + C \| v^0_i \|_{H^1(\Omega^0_i)}.
  \end{aligned}
\end{equation*}
Using this, we can conclude that the third sum of $a$ satisfies
\begin{equation}
  \label{eq:bdd_b}
  \begin{aligned}
    & \sum_{(i, j) \in \E_0} \scalar{B_j (v^0_i - v^1_j)}{w^0_i - w^1_j}_{\Omega^1_j} \\
    & \qquad{} \leq \sum_{(i,j) \in E_0} C \overline{\beta}
        \paren{\| v^0_i \|_{H^1(\Omega^0_i)} + \| v^1_j \|_{H^1(\Omega^1_j)}}
        \paren{\| w^0_i \|_{H^1(\Omega^0_i)} + \| w^1_j \|_{H^1(\Omega^1_j)}} \\
    & \qquad{} \leq C \overline{\beta} \, \| v \|_V \, \| w \|_V,
  \end{aligned}
\end{equation}
where the constant $C$ also depends on the maximum number of interfaces surrounding the subdomains.
\end{proof}

\begin{theorem}
Under the above assumptions, equation \eqref{eq:weak_problem} has a unique solution $u\in V_0$.
\end{theorem}
\begin{proof}
Since $V_0$ is a Hilbert space,  Lemma  \ref{lem:coercivity} and Lemma \ref{lem:boundedness} give coercivity and boundedness of the bilinear form, and the right hand side $F$ is assumed to be a bounded linear functional on $V_0$, the Lax--Milgram theorem guarantees existence of a unique solution $u\in V_0$.
\end{proof}

\section{Fitted finite element method}
\label{sec:fem}

In this section we introduce a fitted finite element discretization of the mixed dimensional model problem and prove an a priori error bound. 

\subsection{Discretization}


We let $\Tau^1_{h,j}$ be a quasi-uniform and shape-regular partition of $\overline{\Omega^1_j}$, containing closed simplices of diameter no greater than $h$, with $\Tau^1_h = \bigcup_{j \in J} \Tau^1_{h,j}$. Further, we let $\Tau^0_{h,i}$ be a corresponding partition of $\overline{\Omega^0_i}$, with $\Tau^0_h = \bigcup_{i \in I} \Tau^0_{h,i}$ and such that the simplices of $\Tau^1_h$ constitute the sides of some of the simplices in $\Tau^0_h$. That is, if $\Tau^1_h = \{K^1_{\hat{\jmath}}\}_{\hat{\jmath}}$ and $\Tau^0_h = \{K^0_{\hat{\imath}}\}_{\hat{\imath}}$, then for each pair $\hat{\imath}, \hat{\jmath}$, either
\begin{itemize}
    \item $K^0_{\hat{\imath}} \cap K^1_{\hat{\jmath}} = \emptyset$,
    \item $K^0_{\hat{\imath}} \cap K^1_{\hat{\jmath}} = K^1_{\hat{\jmath}}$ or
    \item $K^0_{\hat{\imath}} \cap K^1_{\hat{\jmath}} \in \partial K^1_{\hat{\jmath}}$.
\end{itemize}

Similarly to what we did in Subsection \ref{sec:weak_formulation}, for the functions on the interface partition, $\Tau^1_h$, we define 
\begin{equation*}
\begin{aligned}
    &V^1_{{\rm b},h} = \prod_{j \in J} \{ v^1 \in C(\overline{\Omega^1_j}): v^1 \text{ piecewise linear functions on }\Tau^1_{h,j} v^1 = 0 \text{ on } \partial \Omega\}, \text{ and} \\
    &V^1_h = \{v^1 \in V^1_{{\rm b},h}: v^1_j |_{\Omega^2_k} - v^1_{j'} |_{\Omega^2_k} = 0 \text{ for all pairs } (j,k), (j',k) \in \E_1\}.
\end{aligned}
\end{equation*}
For the functions on the bulk partition, $\Tau^0_h$, we define
\begin{equation*}
    V^0_h = \prod_{i \in I} \{ v^0 \in C(\overline{\Omega^0_i}): v^0 \text{ piecewise linear functions on }\Tau^0_{h,i}, v^0 = 0 \text{ on } \partial \Omega\},
\end{equation*}
and use these definitions to define $V_h = V^0_h \times V^1_h$. Then the finite element problem reads: find $u_h \in V_h$ such that
\begin{equation}
    a(u_h, v_h) = F(v_h) \qquad \text{ for all } v_h \in V_h. \label{FE-problem}
\end{equation}

We note that $V^0_h \subset V^0$ and $V^1_h \subset V^1$ are also Hilbert spaces and the argumentation in Section~\ref{sec:well_posedness} holds for the discretized equation as well. Thus, Lax--Milgram guarantees the existence and uniqueness of a solution to \eqref{FE-problem}.



\subsection{Interpolation estimates}

In this section, we introduce an interpolation operator
$\I_h : V_0 \to V_h$ defined componentwise $\I_h(v^0, v^1) = (\I^0_h v^0, \I^1_h v^1)$,
where $\I^0_h$ and $\I^1_h$ are Scott--Zhang \cite{ScottZhang} interpolation operators 
onto $V^0_h$ and $V^1_h$, respectively.

For each degree of freedom $k$ (enumerating the Lagrange basis functions
$\{\varphi^0_k\}$ of $V^0_h$), we choose an edge (or face if $d=3$, etc.)\ $\tilde K_k$ in the support of $\varphi^0_k$ such that the mesh vertex associated with $k$ is contained in $\tilde K_k$. Let $\psi^0_k$ be the $L^2(\tilde{K}_k)$-dual basis, i.e.\
\begin{equation*}
  \int_{\tilde{K}_k} \varphi^0_{k'} \psi^0_{k} \dif x = \begin{cases}1 & \text{ if } k=k', \\ 0 & \text{ otherwise,}\end{cases}
\end{equation*}
and define
$$\I^0_h v^0 = \sum_{k} \left(\int_{\tilde{K}_k} v^0 \psi^0_{k} \dif
  x\right) \varphi^0_k.$$ The interface interpolation operator $\I^1_h$
is defined analogously.

In the $V$-norm, we obtain the interpolation error estimate, for $v \in V_0$,
\begin{equation}
  \norm{v - \I_h v}_V 
  \leq Ch\paren{\norm{D^{2} v^0}_{L^2(\Omega^0)} + \norm{D^{2} v^1}_{L^2(\Omega^1)}}.
  \label{eq:interpolation}
  \end{equation}
For proofs and further details on this interpolation, we refer to \cites[Section 4.8]{BrennerScott}{ScottZhang}.

\subsection{A priori error estimate}
In order to derive an a priori estimate, we use coercivity and boundedness (Lemma~\ref{lem:coercivity} and \ref{lem:boundedness}) together with Galerkin orthogonality. For $v \in V_h$, it holds that
\begin{equation*}
	\begin{aligned}
		\frac{1}{C(\underline{\alpha}^{-1} + \underline{\beta}^{-1})} \norm{u - u_h}_V^2 &
		\leq a(u - u_h, u - u_h) = a(u - u_h, u - v) \\
		&\leq \tilde{C} \paren{\overline{\alpha} + \overline{\beta}} \, \| u - u_h \|_V \, \| u - v \|_V.
	\end{aligned}
\end{equation*}
Thus we obtain
\begin{equation}
	\norm{u - u_h}_V \leq C 
	\paren{ \underline{\alpha}^{-1} + \underline{\beta}^{-1} }
	\paren{ \overline{\alpha} + \overline{\beta}}
	 \norm{u - v}_{V}.
	 \label{eq:Cea}
\end{equation}
By choosing $v = \I_h u$, and combining \eqref{eq:interpolation} and \eqref{eq:Cea}, we can thus state the following theorem.
\begin{theorem}{(A priori error estimate.)}
Let $u_h \in V_h$ and $u \in V$ be the solutions of \eqref{FE-problem} and \eqref{eq:weak_problem}. Then
\begin{equation*}
\begin{aligned}
	\norm{u - u_h}_V
	& \leq C h \paren{ \underline{\alpha}^{-1} + \underline{\beta}^{-1} }
	\paren{ \overline{\alpha} + \overline{\beta}}
	\paren{\norm{D^{2} u^0}_{L^2(\Omega^0)} + \norm{D^{2} u^1}_{L^2(\Omega^1)}}.
\end{aligned}\end{equation*}
\label{thm:apriori}
\end{theorem}



\section{Iterative solution based on subspace decomposition}\label{sec:iterative}

Since the bulk domains are only connected through the interfaces it is natural to consider a Schur complement formulation. We propose an iterative solver of the Schur complement equation that is based on subspace decomposition. The subspaces are introduced using an artificial coarse mesh on top the computational domain, see Figure \ref{fig:implementation_domain2} for an illustration and \cite{KY16,GoHeMa22} for more details. Given the subdomains we use an additive Schwarz preconditioner to solve for the Schur complement. Given the solution on the interfaces we solve for the decoupled bulk regions, using a direct solver in parallel.

\subsection{Schur complement}

We let $\{\varphi^0_{k}\}$ and $\{\varphi^1_{\ell}\}$ be the standard Lagrange bases of $V^0_h$ and $V^1_h$ respectively. Thus the finite element solution $u_h \in V_h$ consists of the two parts $u^0_h = \sum_{k} U^0_{k}  \varphi^0_{k} $ and $u^1_h = \sum_{\ell} U^1_{\ell} \varphi^1_{\ell}$. We denote by $U_0$ and $U_1$ the vectors $(U^0_k)_k$ and $(U^1_\ell)_\ell$, respectively.

Let $A$ and $b$ be the resulting matrix and load vector from \eqref{FE-problem}, using the above mentioned basis. Then the system $AU = b$ can be divided into
\begin{equation}
  \label{eq:matrix}
	\left[
	\begin{array}{c | c}
		A_{00} & A_{01} \\ \hline
		A_{10} & A_{11}
	\end{array}
	\right]
	\begin{bmatrix}
		U_0 \\ U_1
	\end{bmatrix} = 
	\begin{bmatrix}
		b_0 \\ b_1
	\end{bmatrix},
\end{equation}
where $A_{00}$ and $A_{11}$ correspond to the degrees of freedom in the bulk and interfaces, respectively. The submatrices $A_{01}$ and $A_{10} = A_{01}^\top$ describe the connections between the interfaces and bulk regions. We form the Schur complement of the submatrix $A_{00}$ and note that $U_1$ solves:
\begin{equation}\label{eq:schur}
\tilde A_{11}U_1:=\big(A_{11} - A_{10}A_{00}^{-1}A_{01}\big) U_1 = b_1 - A_{10}A_{00}^{-1} b_0=:\tilde b_1.
\end{equation}
The submatrix $A_{00}$ is block diagonal which means that linear systems involving this matrix can be solved cheaply in parallel. We therefore focus on how to solve the Schur complement equation (\ref{eq:schur}) efficiently.




\subsection{Preconditioner based on subspace decomposition}\label{subsection:subspace_decomp}
The goal is to define a preconditioner for equation (\ref{eq:schur}). Inspired by \cite{GoHeMa22, KY16}, we introduce an artifical quasi-uniform coarse ($H>h$) finite element mesh $\Tau_H$ of the computational domain with corresponding finite element space $W_H$, see Figure \ref{fig:implementation_domain2}. We let $\{\phi_j\}_{j=1}^n$ be the set of Lagrangian basis functions spanning $W_H$. The mesh $\Tau_H$ is independent of the meshes $\Tau^0_h$ and $\Tau^1_h$ and thereby independent of the location of the interfaces. Next, we let $\mathcal{I}_h^\text{nodal}$ be the nodal interpolant onto $V^1_h$, which is the finite element space for the interfaces on which the Schur complement equations are posed. We define the coarse space
$$
W_0=\mathcal{I}_h^\text{nodal} W_H\subset V^1_h.
$$
Furthermore, we define for $j = 1, \ldots, n$,
$$
W_j=\{v\in V^1_h:\text{supp}(v)\subset \text{supp}(\phi_j)\}\subset V^1_h.
$$
Thereby we have created a subspace decomposition of $V^1_h$, i.e.\ any $v\in V^1_h$ can be written as a sum $v=\sum_{j=0}^n v_j$ with $v_j\in W_j$.


In the remainder of this section, we use matrix representations
for functions and operators on the interfaces. The functions
$\{\varphi^1_\ell\}_{\ell=1}^m$ are used as basis in $V^1_h$ and the
matrices $Q_j \in \R^{m \times m_j}$ are prolongation
matrices mapping from functions in $W_j$ to $V^1_h$. As basis in $W_j$, for $j = 0$, the
functions $\{\phi_j\}_{j=1}^n$ are used (i.e.\ $m_0 = n$), while for $j \ge 1$, we choose the smallest
subset of $\{\varphi^1_\ell\}_{\ell=1}^m$ spanning $W_j$, whose size we denote by $m_j$.

We now introduce the matrices
$$
T_j = (Q_j^\top \tilde A_{11} Q_j )^{-1} Q_j^\top  \in \R^{m_j \times m}
$$
and form the full preconditioner by the sum,
\begin{equation}\label{eq:pre}
T=\sum_{j=0}^n Q_j T_j \in \R^{m \times m}.
\end{equation}
The matrix $T$ is a preconditioner to \eqref{eq:schur}, $$T\tilde{A}_{11}U_1=T\tilde b_1.$$
Note that in order to compute one application of the preconditioner $T$ we need to solve one global coarse scale problem (for $j = 0$, on the subspace $W_0$) and $n$ independent local problems (for $j \ge 1$, on the subspaces $W_j$).  All these solves are done using a direct solver, which means that the technique we propose is semi-iterative. Under some assumptions of the geometry of the interfaces, it is actually possible to show optimal convergence properties of this preconditioner, in the meaning that the convergence rate is independent of the fine scale mesh size $h$ and the coarse scale mesh size $H$. 

\subsection{Convergence of the preconditioned conjugate gradient (PCG) method}

This preconditioner has been thoroughly analyzed for spatial network problems in \cite{GoHeMa22}. In order to apply the convergence analysis presented there we need to formulate our Schur complement equation as an equation posed on a graph. We limit ourselves to the case $d=2$, so that the interfaces are one-dimensional. We let $\mathcal{N}$ be the set of nodes in the interface finite element mesh $\Tau^1_h$ and let $\mathcal{E}$ be the set of edges connecting two nodes, where an edge connecting two nodes $x$ and $y$ is represented by an unordered pair $\{x, y\}$. This forms the graph $\mathcal{G}=(\mathcal{N},\mathcal{E})$. In \cite{GoHeMa22}, two operators are used in the analysis: a weighted graph Laplacian $L$ and a diagonal mass matrix $M$. They are both symmetric and thereby uniquely defined by their corresponding quadratic forms:
\begin{equation*}
  \begin{aligned}
    (L v_1,v_1)&=\frac{1}{2}\sum_{x\in \mathcal{N}}\sum_{\{x,y\}\in \mathcal{E}}\frac{(v_1(x)-v_1(y))^2}{|x-y|},\\
    (M v_1,v_1)&=\frac{1}{2}\sum_{x\in \mathcal{N}}\sum_{\{x,y\}\in \mathcal{E}}v_1(x)^2 |x-y|,
  \end{aligned}
\end{equation*}
for all $v_1 \in V^1_h$, using the notation $(v,w)=\sum_{x\in \mathcal{N}}v(x)w(x)$. In the following, we overload the notation for $L$, $M$, $v \in V_h$ and $v_1 \in V^1_h$ with their matrix representations compatible with \eqref{eq:matrix}.
\begin{theorem}\label{thm:iterative}
If $d = 2$ and the Poincar\'{e} type inequality $v_1^\top M v_1 \leq D v_1^\top L v_1$ holds for all $v_1\in \mathbb{R}^m$ and some $D < \infty$, then 
there exist constants $c_1,c_2>0$ such that
$$
c_1 v_1^\top L v_1 \leq v_1^\top \tilde A_{11}v_1\leq c_2 v_1^\top L v_1.
$$
\end{theorem}
\begin{proof}
  We start with the second inequality. From the construction of $A_{11}$ it follows that
$$
\underline\alpha v_1^\top L v_1 \leq v_1^\top A_{11}v_1\leq \overline\alpha v_1^\top L v_1+  \overline\beta v_1^\top M v_1.
$$
Since $A_{00}$ is symmetric and positive definite, $v_0^\top A_{00}^{-1}v_0>0$ and therefore \\$v_1^\top A^\top_{01}A_{00}^{-1}A_{01} v_1\geq 0$.
We conclude
$$
v_1^\top \tilde A_{11} v_1\leq \overline\alpha v_1^\top L v_1+  \overline\beta v_1^\top M v_1\leq (\overline\alpha+D\overline\beta) v_1^\top L v_1=: c_2v_1^\top L v_1.
$$

To prove the first inequality, we use that for $v \in V_h$,
$$v^\top A v  \geq \frac{1}{C(\underline{\alpha}^{-1} + \underline{\beta}^{-1})}\norm{v}_V^2=:c_1 \norm{v}_V^2.$$ Specifically, it holds for $v = [-A_{00}^{-1}A_{01}v_1, v_1]^\top$. Thus,
$$
v_1^\top (A_{11} - A_{10}A_{00}^{-1}A_{01})v_1  = v^\top A v \geq
	c_1 \norm{v}_V^2 \geq 
	c_1 v^\top_1 L v_1.
$$

\end{proof}

The assumption $v_1^\top M v_1 \leq D v_1^\top L v_1$ is fulfilled under assumptions of locality, homogeneity, and connectivity of the graph $\mathcal{G}$ on a scale $R_0$.  Locality means that all the edges of the network are shorter than $R_0$, homogeneity means that the variations in total length of the subnetwork contained on a square with side length greater than or equal to $2R_0$, is bounded by a constant, and connectivity means that for any subnetwork contained in a square of size $R\geq R_0$, there is a connected subnetwork contained in an extended square of side length $2R+2R_0$. The precise assumptions can be found in Assumption 3.4 in \cite{GoHeMa22}. In short, the theory is applicable if the interfaces are connected and sufficiently dense in the computational domain. Note that the preconditioner is still applicable and may give rapid convergence even if these assumptions are not fulfilled. 

Under Assumption 3.4 in \cite{GoHeMa22}, Theorem \ref{thm:iterative} holds. Moreover, with
 $H\geq R_0$, Theorem 4.3 in \cite{GoHeMa22} guarantees convergence of the conjugate gradient method for solving $\tilde A_{11}U_1=\tilde b_1$ with preconditioner $T$, defined by equation (\ref{eq:pre}). We let $U_1^{(\ell)}$ denote the approximate solution after $\ell$ iteration and let $|v|^2_{\tilde{A}_{11}}=v^T\tilde{A}_{11}v$.
Then the following error bound holds, 
 $$
|U_1-U_1^{(\ell)}|_{\tilde{A}_{11}}\leq 2\left(\frac{\sqrt{\kappa}-1}{\sqrt{\kappa}+1}\right)^{\ell}|U_1-U_1^{(0)}|_{\tilde{A}_{11}}, 
$$
where $\kappa$ is the condition number of $T\tilde{A}_{11}$.
Expressed in the $V^1$-norm, we  have
$$
\|u^1_h-u^{1,{(\ell)}}_h\|_{V^1} \leq C\left(\frac{\sqrt{\kappa}-1}{\sqrt{\kappa}+1}\right)^{\ell}\|u^1_h-u^{1,{(0)}}_h\|_{V^1},
$$
where $u^{1,{(\ell)}}_h \in V^1_h$ is the function corresponding to the vector $U_1^{(\ell)}$. We note that $C$ only depends on upper and lower bounds of the coefficients in equation (\ref{eq:weak_problem}), and have from \cite{GoHeMa22} that $\kappa$ is independent of the mesh sizes $h$ and
$H$. Once the interface component $U_1^{(\ell)}$ has been computed,
the bulk component $U_0^{(\ell)}$ can be solved from
$$A_{00} U_0^{(\ell)} = b_0 - A_{01} U_1^{(\ell)}.$$ This problem is
easy to solve since $A_{00}$ is block diagonal and each block
corresponds to a problem in a single bulk subdomain segment.

\section{Numerical examples}\label{sec:num_ex}

We start by investigating the convergence theoretically and numerically for the proposed method for some example problems. We then examine the number of iterations needed for the preconditioned conjugate gradient method to converge, for two different test cases. All numerical examples are two dimensional, $d=2$, and posed on the unit square.

\subsection{Regularity and convergence}
\label{sec:num_reg_conv}
The solutions $u^0_i$ on the bulk domains $\Omega_i^0$ to equation \eqref{eq:weak_problem} have $H^{\sfrac{3}{2}}$-regularity. If the bulk domain is convex it has $H^2$-regularity. We formulate this result as a theorem and leave the proof to Appendix \ref{app:regularity}.

\begin{theorem}
Assume $A_i \in C^\infty(\overline{\Omega^0_i})$, $A_j \in C^\infty(\overline{\Omega^1_j})$,
$B_j \in C^\infty(\Omega^1_j)$, $f_i \in L^2(\Omega_i^0)$ and $g_i \in H^{1}(\overline{\Omega_i^0} \cap \partial \Omega)$ for all $i \in I$ and $j \in J$. Then $u_i^0 \in H^{\sfrac{3}{2}}(\Omega_i^0)$ and $u_j^1\in H^2(\Omega_j^1)$. If all subdomains are convex, then additionally $u_i^0 \in H^{2}(\Omega_i^0)$. 
\label{regularity_bulk}
\end{theorem}

In order to ensure that the method converges according to theory, we solve a few problems with several different mesh sizes $h$. The problems we choose for this analysis are a domain with eight infinite interfaces and a domain with 27 finite interfaces, see Figure \ref{fig:implementation_domain}. Infinite interfaces here means straight lines that pass through the entire domain, leading to convex subdomains, while finite lines leads to non-convex subdomains. For all of these cases, we let $A_i = 1$, $A_j = 1$, $B_j = 1$ and $f = e^{-10 \sqrt{(x_1 - \frac{1}{2})^2 + (x_2 - \frac{1}{2})^2}}$ on both the interfaces and the bulk areas. The mesh edges are split in two in each refinement of the mesh. The solution using eight interfaces is presented in Figure \ref{fig:solution}.

The energy norm $a(u_{\text{ref}}, u_{\text{ref}})^{1/2}$ of the reference solution $u_{\text{ref}}$ is then calculated for each setup and compared to the finite element approximations using different levels of refinement. The initial mesh is fitted to the interfaces and minimal in the sense that interfaces are not refined. We then refine the initial mesh uniformly, using regular refinement, four times. The reference solution is constructed using five uniform refinements. The convergence result is presented in Figure \ref{fig:convergence_energy_norm}. We detect a linear convergence rate, which agrees with the theory for the case of convex subdomains, and a convergence rate that exceeds the predicted one for the case of non-convex subdomains.




\begin{figure}[h!]
\begin{subfigure}{0.45\textwidth}
	\includegraphics[width=\textwidth]{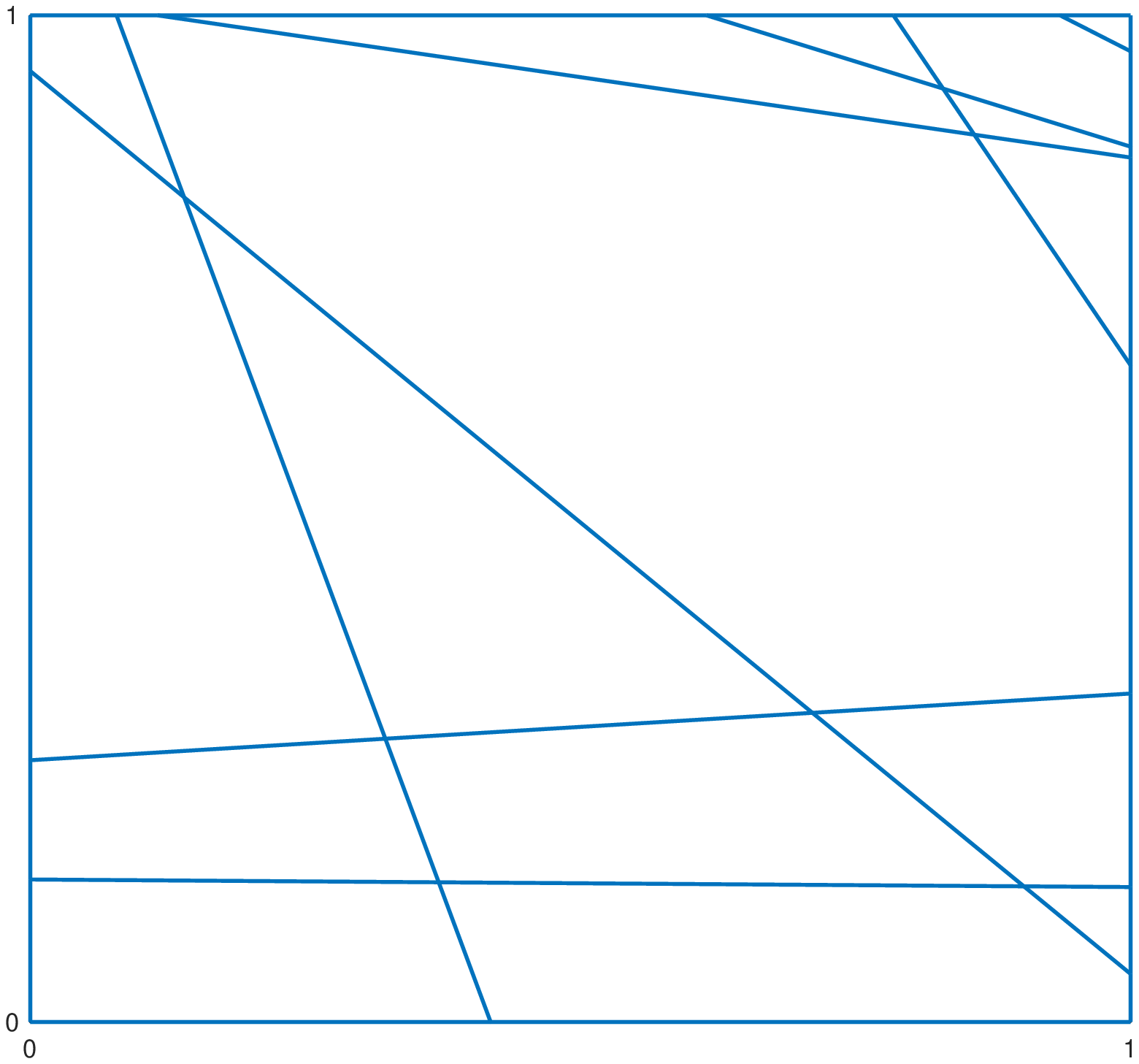}
	\caption{Infinite interfaces.}
\end{subfigure}
\begin{subfigure}{0.45\textwidth}
	\includegraphics[width=\textwidth]{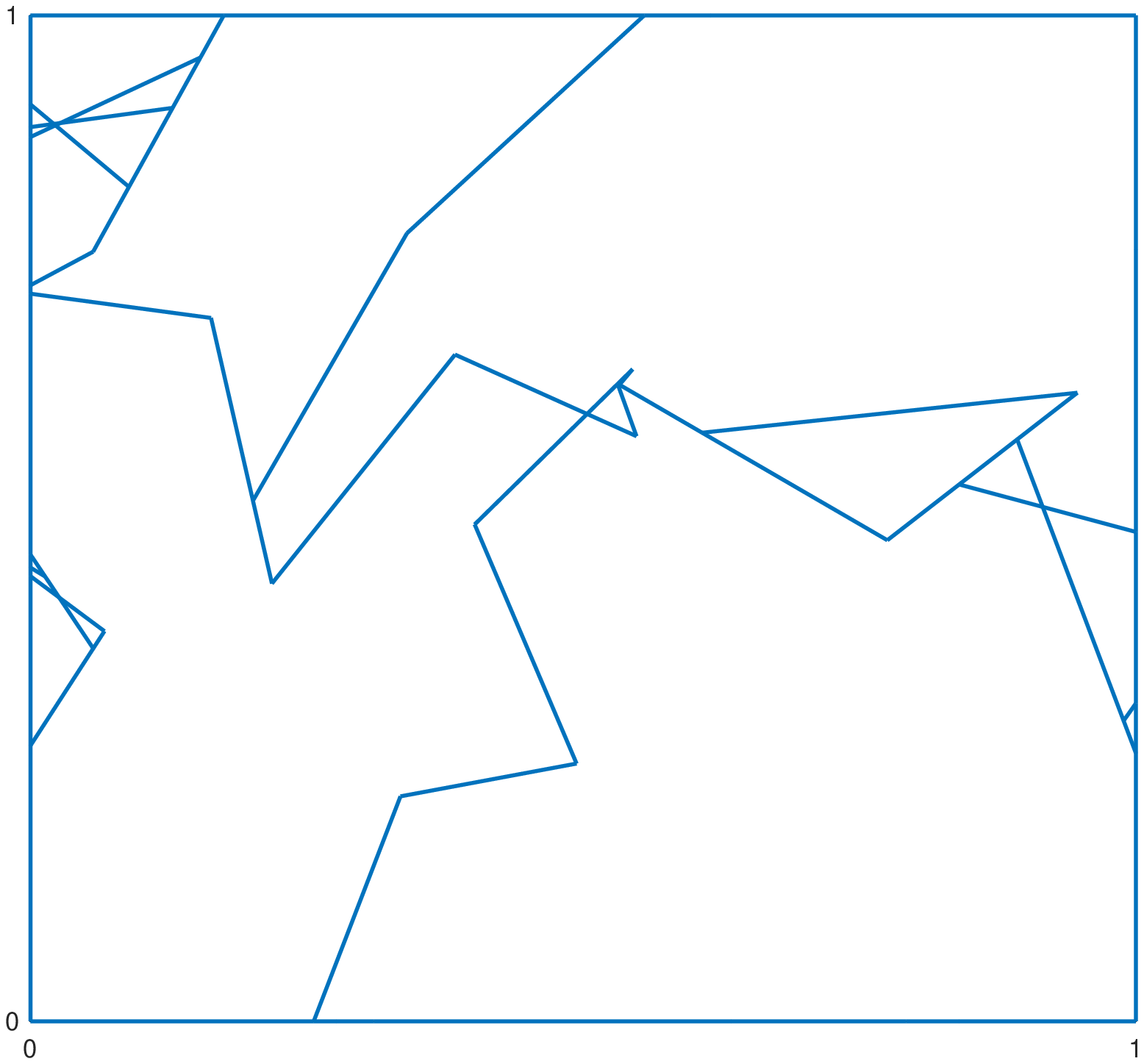}
	\caption{Finite interfaces.}
\end{subfigure}
\caption{Domains in Section~\ref{sec:num_reg_conv} on which the convergence analysis is performed.}\label{fig:implementation_domain}
\end{figure}

\begin{figure}[h!]
	\includegraphics[scale=.45]{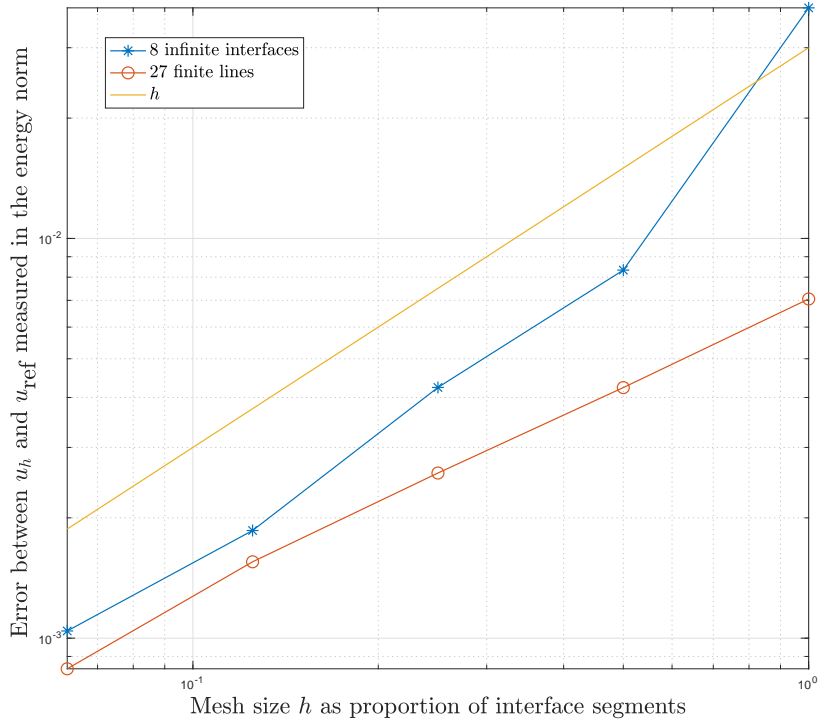}
	\caption{The energy norm of the error between approximations $u_h$ and the reference solution $u_\text{ref}$.}
	\label{fig:convergence_energy_norm}
\end{figure}

\begin{figure}[h!]
	\includegraphics[scale=.6]{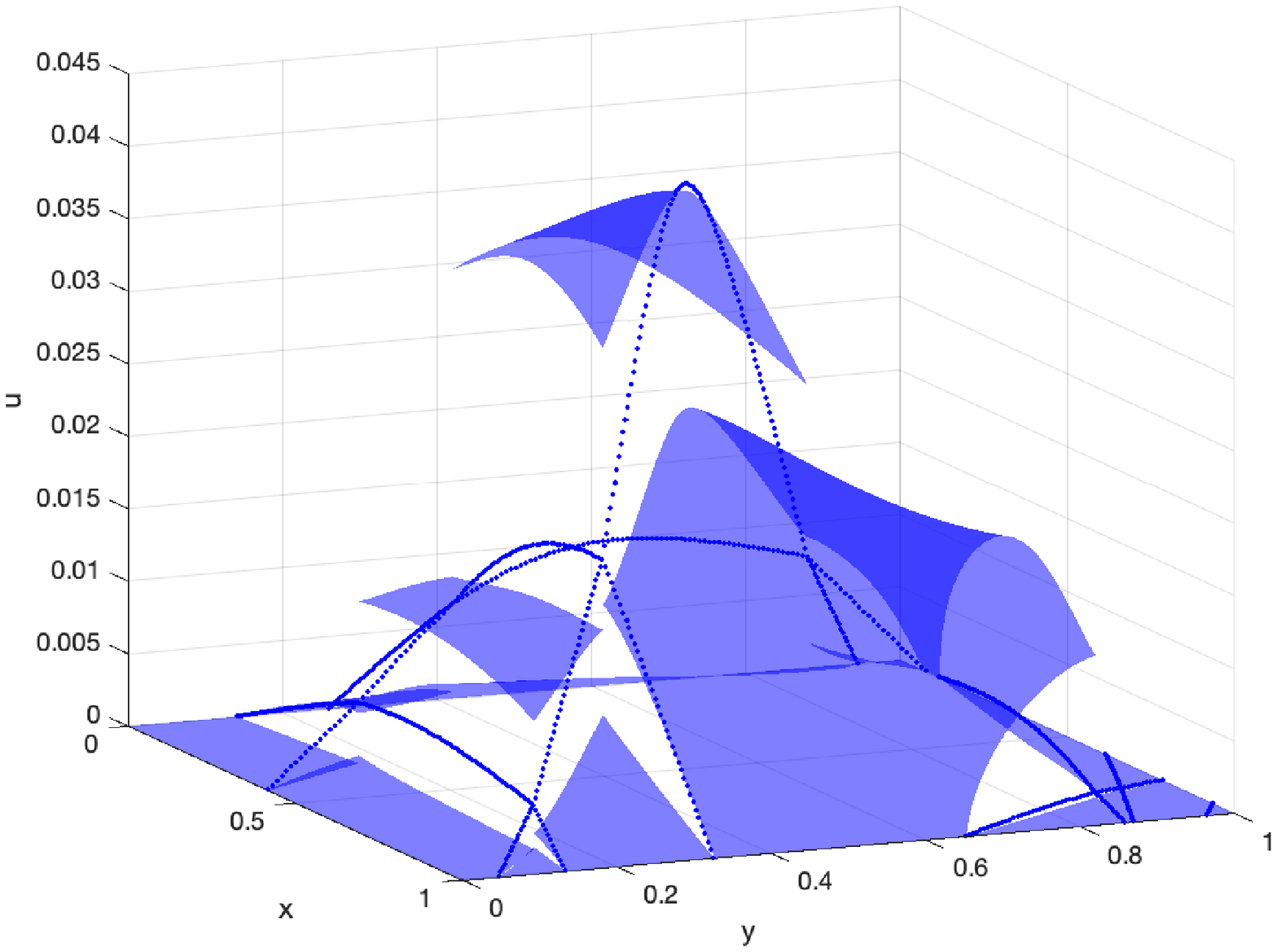}
	\caption{The reference solution $u_\text{ref}$ to the first numerical example with eight interfaces.}
	\label{fig:solution}
\end{figure}

\subsection{Convergence of the iterative solver}
\label{sec:num_conv_iter}
We solve equation \eqref{eq:weak_problem} on a unit square with roughly 200 infinite interfaces, using $A_i = 1$ and $f = e^{-10 \sqrt{(x_1 - \frac{1}{2})^2 + (x_2 - \frac{1}{2})^2}}$ on both the interfaces and the bulk areas, see Figure~\ref{fig:implementation_domain2} for an illustration. The effect of varying the coefficients $A_j$ and $B_j$ are tested for two and three different cases respectively. The coefficients $B_j$, describing the coupling between the bulk and interface subdomains, either all have the value $0.01$ (weakly coupled), $1$ (moderately coupled), or $100$ (strongly coupled). The coefficients $A_j$ either all have the value 1, or are uniformly distributed random numbers in the interval $[0.01, 1]$ on the different interface mesh edges.
For these setups, the condition numbers of the full matrix varies in the range $10^{11}$ to $10^{15}$.
Therefore, solving the system iteratively, without a preconditioner, leads to very poor convergence (around 5\,000 -- 10\,000 iterations in the presented examples).

The number of iterations needed to reach convergence using the preconditioner presented in Section \ref{sec:iterative} for $H=1/8$ (the grid in Figure~\ref{fig:implementation_domain2}) and $H=1/16$ are shown in Table \ref{table:pcg_infinite}. We see rapid convergence of the conjugate gradient method given the high condition number for the initial system matrix. We note that a higher number of iterations for the case $B_j=100$, corresponding to strong coupling, is needed. The convergence rate seems to be independent of, or only depend weakly on $H$ and $h$, as suggested in the theory. The dense distribution of interfaces and the rapid convergence suggests that Theorem \ref{thm:iterative} is applicable. It is however difficult to prove this statement since it depends on the exact connectivity and density properties on the network of interfaces. 


We also solve the problem on a domain with finite lines, leading to non-convex subdomains. We do this for roughly 200 lines of length $\leq 0.2$. The parameters are the same as in the case of infinite interfaces, so we examine weakly, moderately and strongly coupled cases where $A_j$ is either constantly 1 och uniformly distributed between 0.01 and 1. The condition numbers of the full matrix vary in the range $10^8$ to $10^{13}$. The number of iterations using a preconditioner are shown in Table \ref{table:pcg_finite}. 
Again, we notice that the preconditioner does a very good job and that the convergence rate is independent or only weakly dependent on the method parameters $H$ and $h$. We again need a higher number of iterations for the case $B_j=100$ corresponding to strong coupling. 

\begin{figure}[h!]
\begin{subfigure}{0.49\textwidth}
	\includegraphics[width=\textwidth]{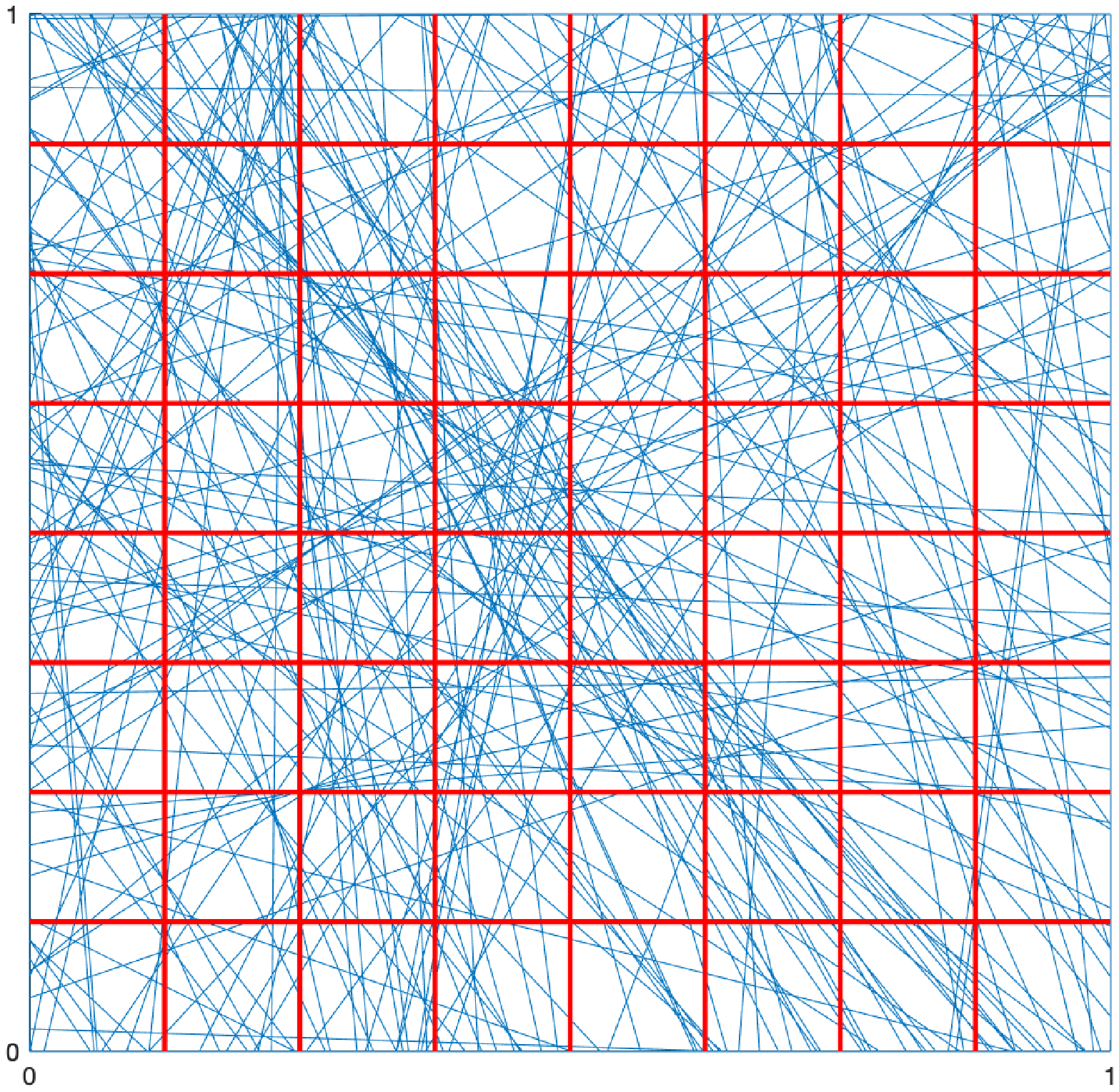}
	\caption{Infinite interfaces}
\end{subfigure}
\hfill
\begin{subfigure}{0.49\textwidth}
	\includegraphics[width=\textwidth]{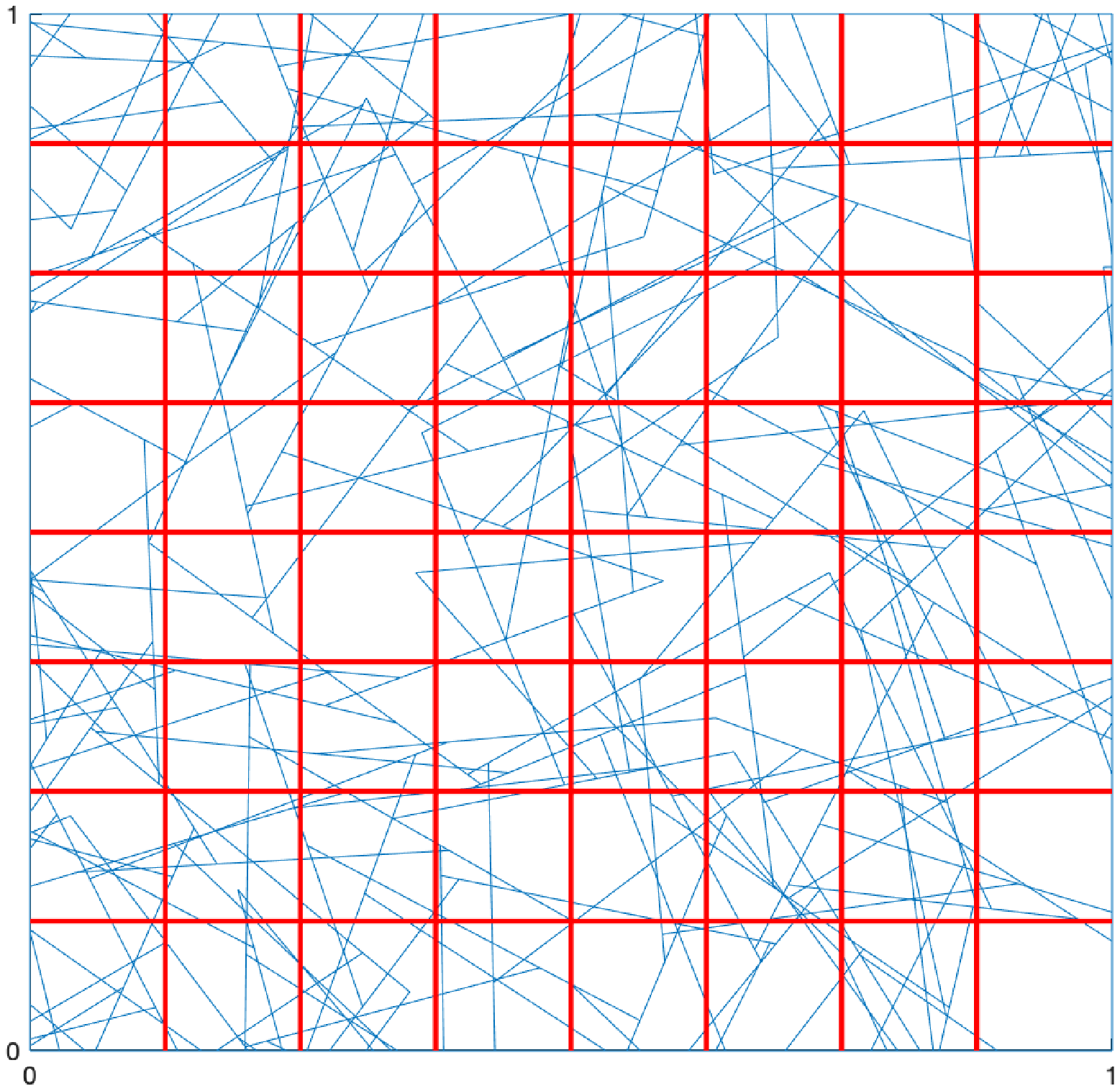}
	\caption{Finite intefaces}
\end{subfigure}
\caption{Domains in Section~\ref{sec:num_conv_iter} using coarse grid mesh size $H=1/8$ for the subspace correction.}\label{fig:implementation_domain2}
\end{figure}


\begin{table}[h!]
\caption{Number of iterations until convergence using the preconditioner, on a domain with approximately 200 infinite interfaces.}\label{table:pcg_infinite}
\begin{subtable}[h]{.48\textwidth}
\caption{$H = 1/8$}
	\begin{tabular*}{\textwidth}{@{\extracolsep{\fill}} | l | r | r |} \hline 
	\diagbox{$B_j$}{$A_j$} & $[0.01, 1]$ & 1 \\ \hline
	0.01 		& 31 		& 25		 \\
	\hline
	1 		& 33 		& 26		 \\
	\hline
	100 		& 73 		& 50		 \\
	\hline
	\end{tabular*}
\end{subtable}
\hfill
\begin{subtable}[h]{.48\textwidth}
\caption{$H = 1/16$}
	\begin{tabular*}{\textwidth}{@{\extracolsep{\fill}} | l | r | r |} \hline 
	\diagbox{$B_j$}{$A_j$} & $[0.01, 1]$ & 1 \\ \hline
	0.01 		& 36 		& 27		 \\
	\hline
	1 		& 39 		& 28		 \\
	\hline
	100 		& 87 		& 60		 \\
	\hline
	\end{tabular*}
\end{subtable}
\end{table}

\begin{table}[h!]
\caption{Number of iterations until convergence using the preconditioner, on a domain with approximately 200 finite interfaces.}\label{table:pcg_finite}
\begin{subtable}[h]{.48\textwidth}
\caption{$H = 1/8$}
	\begin{tabular*}{\textwidth}{@{\extracolsep{\fill}} | l | r | r |} \hline 
	\diagbox{$B_j$}{$A_j$} & $[0.01, 1]$ & 1 \\ \hline
	0.01 		& 29 		& 24		 \\
	\hline
	1 		& 31 		& 25		 \\
	\hline
	100 		& 65 		& 46		 \\
	\hline
	\end{tabular*}
\end{subtable}
\hfill
\begin{subtable}[h]{.48\textwidth}
\caption{$H = 1/16$}
	\begin{tabular*}{\textwidth}{@{\extracolsep{\fill}} | l | r | r |} \hline 
	\diagbox{$B_j$}{$A_j$} & $[0.01, 1]$ & 1 \\ \hline
	0.01 		& 36 		& 28		 \\
	\hline
	1 		& 37 		& 29		 \\
	\hline
	100 		& 75 		& 54		 \\
	\hline
	\end{tabular*}
\end{subtable}

\end{table}


%

%

\subsection*{Acknowledgements}

The second author is supported by the Swedish Research Council project number 2019-03517.

\printbibliography

\begin{appendices}
\section{Regularity}\label{app:regularity}

In this section, we prove that the regularity of a solution to equation (\ref{eq:weak_problem}) posed on a two dimensional polygonal Lipschitz domain is $H^{\sfrac{3}{2}}$ in the bulk regions in general and $H^2$ in case of convex subdomains. We also show that we have $H^2$-regularity on the interface segments. 


The strong formulation of the problem posed on $\Omega_i^0$ reads
\begin{equation*}
\begin{aligned}
	- \div A_i \nabla u_i^0 &= f_i, && \text{ in } \Omega^0_i, \\
	u_i^0 &= g_i, &&\text{ on } \overline{\Omega_i^0} \cap \partial \Omega, \\
	n_{\Omega_i^0} \cdot A_i \nabla u_i^0 + B_j u_i^0 &= B_j u_j^1, && \text{ on } \partial \Omega_i^0 \cap \Omega_j^1.
\end{aligned}
\end{equation*}
With the help of the product rule and the fact that $A_i \geq \underline{\alpha} > 0$, this problem can be rewritten as
\begin{equation}
\begin{aligned}
	- \Delta u_i^0 &= \frac{f_i + \nabla A_i \cdot \nabla u_i^0}{A_i}, && \text{ in } \Omega^0_i, \\
	u_i^0 &= g_i, &&\text{ on } \overline{\Omega_i^0} \cap \partial \Omega, \\
	n_{\Omega_i^0} \cdot \nabla u_i^0 + \frac{B_j}{A_i} u_i^0 &= \frac{B_j}{A_i} u_j^1, && \text{ on } \partial \Omega_i^0 \cap \Omega_j^1.
	\label{eq:bulkproblem}
\end{aligned}
\end{equation}

We allow the subdomain $\Omega^0_i$ to be non-convex, but assume $\Omega$ to be convex. We let $\Gamma_k$ denote the segments of the boundary of $\Omega^0_i$. We let $\omega_k$ denote the internal angle at corner $P_k$, between $\Gamma_k$ and $\Gamma_{k+1}$. Further, we let $D$ be the set of indices of boundary segments that have Dirichlet boundary conditions, i.e. the segments from $\partial \Omega$. Similarly, we let $R$ be the set of indices of the boundary segments equipped with Robin boundary conditions. We let $S$ be the set of indices $k$ for which $\Gamma_k$ and $\Gamma_{k+1}$ both belong to either $D$ or $R$, and $M$ be the set of indices $k$ for which $\Gamma_k$ and $\Gamma_{k+1}$ belong to one each of them.

We have that $\omega_k < \pi$, $k \in D \cup M$, because of our assumption that $\Omega$ is convex. We also have that $\omega_k < 2 \pi$ for all other $k$. In Figure \ref{fig:omega_i} we visualize an example of a subdomain $\Omega^0_i$.
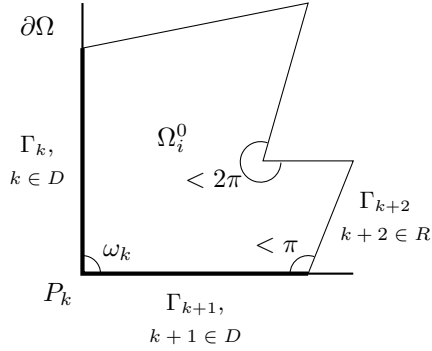
\begin{figure}
\begin{centering}
\begin{tikzpicture}[scale=3]
	\draw [thick] (0,1.2) -- (0,0) -- (1.2,0);
	\draw [ultra thick] (0,1) -- (0,0) -- (1,0);
	\draw (0,1) -- (1,1.2) -- (0.8,0.5) -- (1.2,0.5) -- (1,0);
	\node at (-0.2,1.1) {$\partial \Omega$};
	\node [align=center] at (-.2, .5) {\small $\Gamma_{k},$\\ \scriptsize $k \in D$};
	\node [align=center] at (.5, -.2) {\small $\Gamma_{k+1},$\\ \scriptsize $k+1 \in D$};
	\node [align=center, right] at (1.1, .25) {\small $\Gamma_{k+2}$\\ \scriptsize $k+2 \in R$};
	\node [align=center, below left] at (0,0) {$P_{k}$};
	\draw (0,0.08) arc [radius = .08, start angle = 90, end angle = 0];
	\node [align=center, above right] at (0.05,0.02) {$\omega_k$};
	\draw (.92,0) arc [radius = .08, start angle = 180, end angle = 70];
	\node [align=center, above left] at (1,0.05) {$< \pi$};
	\draw (.82,0.58) arc [radius = .09, start angle = 70, end angle = 358];
	\node [align=left, below left] at (0.75,0.5) {$< 2 \pi$};
	\node at (.4,.6) {$\Omega^0_i$};
\end{tikzpicture}
\caption{An example of a subdomain $\Omega^0_i$. Angles between different types of boundary conditions have to be smaller than $\pi$, while other angles have to be smaller than $2 \pi$.}
\label{fig:omega_i}
\end{centering}
\end{figure}
We are now ready to prove Theorem \ref{regularity_bulk}.

\begin{proof}
From Lax-Milgram we have that $u^0_i \in H^1(\Omega^0_i)$ and $u^1_j \in H^1(\Omega^1_j)$.
Since $u^0_i \in H^1(\Omega^0_i)$, along with the assumptions that $f_i \in L^2(\Omega^0_i)$ and $A_i$ is smooth and positive, we have that $\paren{f_i + \nabla A_i \cdot \nabla u_i^0}/A_i \in L^2(\Omega^0_i)$. Further, we have that $g_i \in H^{1}(\overline{\Omega_i^0} \cap \partial \Omega)$, $B_j/A_i \in C^\infty(\partial \Omega_i^0 \cap \Omega_j^1)$, $B_j/A_i > \underline{\beta} / \norm{A_i}_{\infty} > 0$ and $B_j u_j^1 /A_i \in L^2(\partial \Omega_i^0 \cap \Omega_j^1)$.
We have that $\omega_k/(2 \pi) \notin \N$ for $k \in S$ and $\omega_k/(2 \pi) + 1/2 \notin \N$ for $k \in M$.
Thus, all the requirements of Theorem 2.1 and Proposition 3.1 from \cite{Mghazli} are fulfilled for our problem \eqref{eq:bulkproblem} and hence there exist constants $\alpha^\ell_k$ depending on $f_i$, $A_i$, $B_j$ and $u_j^1$, such that 
\begin{equation}
	u_i^0 - \sum_{0 < \lambda_{k,\ell} < 1} \alpha^\ell_k u^{(0)}_{k,\ell} \in H^2(\Omega^0_i),
	\label{eq:regularity_sum}
\end{equation}
where
\begin{equation}
	\lambda_{k,\ell} = \begin{cases}
		\frac{\ell \pi}{\omega_k}, &k \in S,\\
		\paren{\ell - \frac{1}{2}}\frac{\pi}{\omega_k}, &k \in M,
	\end{cases}
	\qquad \ell = 1,2,3,...
\end{equation}
and
\begin{equation}
u^{(0)}_{k,\ell}(r, \theta) = 
	\begin{cases}
		r^{\lambda_{k,\ell}} \sin \paren{\lambda_{k,\ell} \theta}, \qquad k \in D \cap S \text{ or } R \cap M,\\
		r^{\lambda_{k,\ell}} \cos \paren{\lambda_{k,\ell} \theta}, \qquad k \in R \cap S \text{ or } D \cap M,
	\end{cases}
\end{equation}
if $\lambda_{k,\ell} \notin \N$, which is the only case present in \eqref{eq:regularity_sum}.
The variables $r$ and $\theta$ are the local polar coordinates such that $P_k = (0,0)$, $\theta = 0$ along $\Gamma_{k+1}$ and $\theta = \omega_k$ along $\Gamma_k$.

For both $k \in S$ and $k \in M$, we have that $\lambda_{k,l} > \frac{1}{2}$. With $\frac{1}{2} < \lambda_{k,l} < 1$, we have that $u_{k,l}^{(0)} \in H^{\sfrac{3}{2}}(\Omega_i^0)$. Thus we conclude that $u_i^0 \in H^{\sfrac{3}{2}}(\Omega^0_i)$.
If all subdomains $\Omega_i^0$ are convex and the requirements of Theorem \ref{regularity_bulk} are fulfilled, we see from equation \eqref{eq:bulkproblem} that $u_i^0 \in H^2(\Omega^0_i)$.




Isolating one interface segment $\Omega_j^1$, one gets the partial differential equation:
\begin{equation*}
\begin{aligned}
  - \div_{\tau} A_j \nabla_{\tau} u^1_j - \sum_{i \,:\, (i,j) \in E_0} B_j(u^0_i - u^1_j) &= f_j, &&\Omega^1_j, \\
	u_j^1 &= g_j,
	&&\overline{\Omega^1_j} \cap \partial \Omega, \\
	u^1_j &= u^1_{j'},
	&&\partial \Omega^1_j \cap \partial \Omega^1_{j'}, \\
	\sum_{(j,k) \in E_1} n_{\Omega^1_j} \cdot A_j \nabla_\tau u^1_j &= 0,
	&&\partial \Omega^1_j \cap \Omega^2_k,
\end{aligned}
\end{equation*}
which, with $\Omega \subset \R^2$, becomes:
\begin{equation*}
\begin{aligned}
	-\frac{\dd^2}{\dd x^2} u_j^1 &= 
	\frac{
		f_j +  \sum_{i \,:\, (i,j) \in E_0} B_j (u_i^0 - u_j^1)+ \frac{\dd}{\dd x} A_j \frac{\dd}{\dd x} u_j^1 
	}{A_j},
	&&\Omega^1_j, \\
	u_j^1 &= g_j,
	&&\overline{\Omega^1_j} \cap \partial \Omega, \\
	u^1_j &= u^1_{j'},
	&&\partial \Omega^1_j \cap \partial \Omega^1_{j'}, \\
	\sum_{(j,k) \in E_1} A_j \frac{\dd}{\dd x} u^1_j &= 0,
	&&\partial \Omega^1_j \cap \Omega^2_k,
\end{aligned}
\end{equation*}
where $x$ is a parametrisation of the interface segment.

Since $f_j \in L^2(\partial \Omega^0_i \cap \Omega^1_j)$, and $B_j$ and $A_j$ are smooth, we can conclude that $u_j^1 \in H^2( \Omega^1_j)$. 

\end{proof}

%
%
%
%
%
%
%
%
%
%
%

\end{appendices}

\end{document}